\documentclass[11pt]{article}

\usepackage{amssymb, amsfonts, amsmath, amsthm} 
\usepackage{youngtab}
\usepackage{changes}
\usepackage{tikz}

\newtheorem{theorem}{Theorem}[section]
\newtheorem{lemma}[theorem]{Lemma}
\newtheorem{proposition}[theorem]{Proposition}

\theoremstyle{definition}
\newtheorem{definition}[theorem]{Definition}
\newtheorem{notation}[theorem]{Notation}
\newtheorem{remark}[theorem]{Remark}

\theoremstyle{remark}

\newtheorem*{claim1}{Claim 1}
\newtheorem*{claim2}{Claim 2}
\newtheorem*{claim3}{Claim 3}
\newtheorem*{claimpf}{Proof of Claim}

\newcommand{\bN}{ {\mathbb N} }

\newcommand{\To}{\Rightarrow}
\newcommand{\p}{ {\mathbf P} }

\newcommand{\bY}{ {\mathbb Y} }
\newcommand{\cut}{ {K} }
\newcommand{\depth}{ {\text{Depth}} }

\newcommand{\ahalf}{ \frac{1}{2} }
\newcommand{\halfchithree}{\frac{\chi_3}{2}}

\begin{document}

$\ $

\begin{center}
{\bf\Large Stabilization time for a type of evolution}

\vspace{6pt}

{\bf\Large on binary strings}

\vspace{20pt}

{\large      Jacob Funk  \footnote{Research supported by an USRA 
                                   Award from NSERC, Canada.}
\hspace{2cm} Mihai Nica  \footnote{Research supported by a 
                                    MacCracken fellowship from New York University}
\hspace{2cm} Michael Noyes}
\end{center}

\vspace{10pt}

\begin{abstract}
We consider a type of evolution on $\{0,1\}^{n}$ which occurs in discrete steps whereby at each step, we replace every occurrence of the substring $``01"$ by $``10"$. After at most $n-1$ steps we will reach a string of the form $11\cdots1100\cdots00$, which we will call a $``$stabilized$"$ string and we call the number of steps required the $``$stabilization time$"$. If we choose each bit of the string independently to be a $1$ with probability $p$ and a $0$ with probability $1-p$, then the stabilization time of a string in $\{0,1\}^{n}$ is a random variable with values in $\{0,1,\ldots n-1\}$. We study the asymptotic behavior of this random variable as $n\to\infty$ and we determine its limit distribution after suitable centering and scaling . When $p \neq \ahalf$, the limit distribution is Gaussian. When $p = \ahalf$, the limit distribution is a $\chi_3$ distribution. We also explicitly compute the limit distribution in a threshold setting where $p=p_n$ varies with $n$ given by $p_n = \ahalf + \frac{\lambda / 2}{\sqrt{n}}$ for $\lambda > 0$ a fixed parameter. This analysis gives rise to a one parameter family of distributions that fit between a $\chi_3$ and a Gaussian distribution. The tools used in our arguments are a natural interpretation of strings in $\{0,1\}^{n}$ as Young diagrams, and a connection with the known distribution for the maximal height of a Brownian path on $[0,1]$.
\end{abstract}  

\section{Introduction}

For $n \in \bN$ and $p \in (0,1)$ let $\Omega_n^p$ denote
the probability space consisting of strings in $\{ 0,1 \}^n$, 
  where each bit  is chosen independently to be a
1 with probability $p$ or a 0 with probability $1-p$. We
consider the following kind of `evolution' for 
$\omega \in \Omega_n^p$: replace every occurrence of the substring $``01"$ with $``10"$.
By doing so, new instances
of $``01"$ may be created (for instance  $0101 \mapsto 1010$ 
creates a $``01"$ in the middle). We repeat this process
until we reach a string of the form
$11 \cdots 1100 \cdots 00$.

A concrete example: say we have $n=8$ and we start with the string
$\omega = 01101011$. Then our evolution produces the following strings before stabilizing:
\[
01101011 \mapsto 10110101 \mapsto 11011010
\mapsto 11101100 \mapsto 11110100 \mapsto 11111000,
\]
and stabilizes after 5 iterations because there are no more instances of $``01"$ to be found.

This evolution has a cute interpretation as a line of confused soldiers, see \cite{Kv80}. There is also a way to view the problem in terms of particles whose motion is restricted to one dimension. Imagine that each $1$ in the string is a particle that would like to move to the left-hand side of the string and that each $0$ is an empty space. At every iteration, if a particle has an empty space to its left, it will move into that space. This is exactly the replacement rule $01\to10$. Equally well, one can think of the $0$'s as the particles which are trying to move as far right as they can and the $1$'s as open spaces. The process will stabilize when all of the particles have moved as far as they can go. In this guise we have a kind of deterministic analogue of certain "exclusion processes" (see for instance \cite{L2005} Chapter 8) except in our case the initial condition is random, but the evolution is deterministic. 

It is elementary to show that this process must stabilize after at most $n-1$ iterations, and that afterwards we will obtain a string of the form $11 \cdots 1100 \cdots 00$.  The number of iterations until we reach such a final configuration is a random variable on the probability space $\Omega_{n}^{p}$. We will call this random variable the \emph{stabilization time} and denote it by $T_{n}^{p} : \Omega_{n}^{p} \to \{0,1,\ldots,n-1\}$. Going back to our concrete example above, we have $T_{8}^{p} \left(01101011\right) = 5$.

In this paper we will examine the limit distribution for the random 
variable $T_{n}^{p}$ in the limit $n\to\infty$ and for varying values of $p$.  Because of symmetry between the zeros and ones in the string, we will consider only the case that $p \geq \ahalf$. The case $p \leq \ahalf$ is completely complementary by the replacement $p \leftrightarrow 1-p$. One of the points of interest is the fact that the limit distribution in the case $p=\frac{1}{2}$ is qualitatively different than for  $p \neq \ahalf$.  Motivated by this, we will also consider the case where $p_n$ depends on $n$ and is given by:

 \[
 p_n = \ahalf + \frac{\lambda / 2}{\sqrt{n}}
 \]
  Where $\lambda > 0$ is a positive parameter. We call this the ``\emph{threshold setting}" because it is somewhere between the case $p = \ahalf$ and $p \neq \ahalf$.  

\begin{theorem}
\label{thm:1.1}
We have the following weak limits for the distribution of the random variable $T_{n}^{p}$ in the limit $n\to\infty$:

In the case $p>\frac{1}{2}$:

\[
\frac{T_{n}^{p}-pn}{\sqrt{n}}\Rightarrow N(0,p(1-p))
\]

In the case $p=\frac{1}{2}$:

\[
\frac{T_{n}^{\ahalf}-\frac{1}{2}n}{\sqrt{n}}\Rightarrow \halfchithree
\]

 In the threshold setting $p_n = \ahalf + \frac{\lambda / 2}{\sqrt{n}}$:

\[
\frac{T_{n}^{p_n}-\ahalf n}{\sqrt{n}}\Rightarrow \nu_\lambda
\]

\noindent In the above limits, $N\left(0,p(1-p)\right)$ is a mean zero Gaussian
variable with variance $p(1-p)$, and $\halfchithree \sim \ahalf \sqrt{Z_1^2 + Z_2^2 + Z_3^2}$ is half of the Euclidean norm of a vector of three independent standard $N(0,1)$ Gaussian variables. This has density:
\[
d\halfchithree(x)= \frac{8\sqrt{2}}{\sqrt{\pi}}x^{2}e^{-2x^{2}}dx \text{ for } x>0
\]
 Finally, $\nu_\lambda$ is a random variable supported on the positive real axis  that depends on the parameter $\lambda$, whose density we find explicitly:
\[
d\nu_\lambda(x)= \frac{4\sqrt{2}}{\lambda \sqrt{\pi}} e^{-\ahalf \lambda^2} \sinh (2 \lambda x) x e^{-2x^{2}} dx \text{ for } x>0
\]

\end{theorem}

\begin{remark}
The distribution $\nu_\lambda$ which is found above is somewhere between a $\halfchithree$ and a Gaussian distribution. One can easily verify that as $\lambda \to 0$ the density function for $\nu_\lambda$ converges pointwise to the density function for $\halfchithree$. On the other hand, if we examine $\frac{T_{n}^{p_n}-p_n n}{\sqrt{n}}$, we see from the theorem that:
\[
\frac{T_{n}^{p_n}-p_n n}{\sqrt{n}} = \frac{T_{n}^{p_n}-\ahalf n}{\sqrt{n}} - \frac{\lambda}{2} \Rightarrow \nu_\lambda - \frac{\lambda}{2}
\]
The random variable $\nu_\lambda - \frac{\lambda}{2}$ has density function $\frac{4\sqrt{2}}{\lambda \sqrt{\pi}} e^{-\ahalf \lambda^2} \sinh \left(2 \lambda (x + \frac{\lambda}{2})\right) (x + \frac{\lambda}{2}) e^{-2(x + \frac{\lambda}{2})^{2}}$ for $x > -\lambda/2$. An easy calculation shows that as $\lambda \to \infty$ this density function converges pointwise to $\sqrt{\frac{2}{\pi}}e^{-2x^2}$. This is exactly the density of a $N(0,1/4)$ random variable. 
\end{remark}

\begin{remark} The proof of the theorem comes through connecting several ideas from combinatorics and probability theory. First, we notice a natural connection between strings in $\Omega_n^p$ and Young diagrams. Using Young diagrams as a tool, we can analyze the special case of strings that begin with a 0 and end with a 1 and relate the stabilization time to a simple random walk in one dimension. Using this connection to the random walk, we can find the limit distribution for the stabilization time in the special case mentioned above using the central limit theorem and Donsker's theorem. Finally, we show that the limit distribution for the special case of strings is actually the same as the limit distribution for general strings.

We have divided the proof into three lemmas which are stated below and each discussed and proved in their own sections. 
\end{remark}

\begin{lemma}
\label{lemma:1.2}
Given a string $\omega \in \{0,1\}^n$ there is a a natural random walk associated with $\omega$ which takes a step up for every $0$ in $\omega$ and takes a step down for every $1$ in $\omega$. To be precise, for $0\leq k\leq n$ let $S_{k}=\sum_{i=1}^{k}(1 - 2\omega_{i})$.
\noindent
In the special case where $\omega_1 = 0$ and $\omega_n = 1$, we have the following explicit relationship between the random walk $S_k$ and the stabilization time $T_{n}^{p}(\omega)$:

\[
T_{n}^{p}(\omega)=\frac{n}{2}+\max_{1\leq k\leq n}S_{k}-\frac{S_{n}}{2}-1
\]

\end{lemma}
\begin{remark}
The proof of this lemma comes by mapping each string to a Young diagram in a natural way. The stabilization time of string turns out to be equal to a quantity called the \emph{depth} of the Young diagram. In the special case $\omega_1 = 0$ and $\omega_n = 1$, the depth of the Young diagram is also found to be equal to the above expression for our random walk. The proof of this lemma is discussed in Section \ref{connection_to_rw_section}.
\end{remark}
\begin{lemma}
\label{lemma:1.3}
Let $\tilde{\Omega}_{n}^{p} = \{\omega \in \{0,1\}^n : \omega_1 = 0, \omega_n = 1\}$ be the probability space where each bit, except for the first and last, is chosen independently at random to be a 1 with probability $p$ and to be $0$ with probability $1-p$. The stabilization time of a string is a random variable $\tilde{T}_{n}^{p}:\tilde{\Omega}_{n}^{p}  \to \{1,2,\ldots,n-1\}$.  Then we have convergence of $\tilde{T}_{n}^{p}$ analogous to the statement in Theorem \ref{thm:1.1}: $(\tilde{T}_{n}^{p} - pn)/\sqrt{n} \To N(0,p(1-p)$ for $p>\ahalf$, $(\tilde{T}_{n}^{p} - \ahalf n)/\sqrt{n} \To \halfchithree$ for $p = \ahalf$ and $(\tilde{T}_{n}^{p_n} - \ahalf n)/\sqrt{n} \To \nu_\lambda$ in the threshold setting $p=p_n$.

\end{lemma}
\begin{remark}
Once $\tilde{T}_{n}^{p}(\omega)=\frac{n}{2}+\max_{1\leq k\leq n}S_{k}-\frac{S_{n}}{2}-1$
is established in Lemma \ref{lemma:1.2}, the proof of Lemma \ref{lemma:1.3} is an exercise using the central limit theorem and results about convergence of random walks to diffusion processes. For $p > \frac{1}{2}$
some simple analysis shows that the term $\max_{1\leq k\leq n}S_{k}$
does not contribute to the limit distribution, and the central limit
theorem is enough to prove the limit. In the setting $p=\frac{1}{2}$, the random variable
$\chi_3$ arises from a connection to Brownian motion. By Donsker's theorem we know that when scaled correctly the random walk $S_k$ converges to a Brownian motion. The $\chi_3$ arises from the following fact: if $B_{t}$ is a Brownian motion
and $M_{t}=\max_{s\leq t}B_{s}$ is its running maximum then $M_{1}-\frac{1}{2}B_{1} \stackrel{d}{=} \halfchithree$.  In the threshold setting, we analogously observe that the random walk $S_k$ converges to a Brownian motion but this time with constant positive drift $\lambda$. We then use the Girsanov theorem to calculate $M^{\lambda}_{1}-\frac{1}{2}B^{\lambda}_{1} \stackrel{d}{=} \nu_\lambda$ when $B^{\lambda}_t$ is a Brownian motion with drift $\lambda$. The full details of the proof of this lemma are displayed in Section 3. 
\end{remark}

\begin{lemma}
\label{lemma:1.6}
Let $\tilde{\mu}_{r}^{p}$ be the law of the special case random variable
$\tilde{T}_{r}^{p}$ and let $\mu_{n}^{p}$ be the law of the random
variable $T_{n}^{p}$. Let $\delta_{0}$ be the unit mass at 0.

For $p = \frac{1}{2}$ these measures are related by:

\[
\mu_{n}^{\frac{1}{2}}=\frac{n+1}{2^{n}}\delta_{0} + \sum_{r=0}^{n-2} \frac{n-r-1}{2^{n-r}} \tilde{\mu}_{r+2}^{\frac{1}{2}}
\]

For $p \neq \frac{1}{2}$ these measures are related by:

\[
\mu_{n}^{p}=\frac{p^{n+1}-(1-p)^{n+1}}{2p-1}\delta_{0}+\sum_{r=0}^{n-2}\frac{(1-p)p^{n-r}-p(1-p)^{n-r}}{2p-1}\tilde{\mu}_{r+2}^{p}
\]

 Moreover, from this
relation, we can show that the random variable $T_{n}^{p}$ will have
the same limit distribution as the special case random variable $\tilde{T}_{n}^{p}$.  The same argument holds in the threshold setting $p=p_n$.
\end{lemma}
\begin{remark}
The formulas stated in Lemma \ref{lemma:1.6} are obtained by carefully counting the number of ways of adding leading 1's on the left and trailing 0's on the right of a string of special type $\tilde{\omega} \in \tilde{\Omega}_n^p$, and by exploiting the obvious fact that such additions of leading 1's and trailing 0's does not affect stabilization time. Hence, the  stabilization time for a general string $\omega$ is equal to the stabilization time for the special case substring $\tilde{\omega}$ one obtains from $\omega$ by deleting any leading $1$'s or trailing $0$'s. This gives the relationship between $\tilde{\mu}$ and $\mu$ in the lemma. One can see that this is a convex combination of the probability measures $\tilde{\mu}_r$, which is heavily weighted towards higher values of $r$. Some elementary analysis is used to show that this weighting is such that $\mu$ and $\tilde{\mu}$ have the same limit distribution. The details of this lemma are given in Section 4. 
\end{remark}

\section{Connection to Random Walk in the Special Case}
\label{connection_to_rw_section}

In this section we aim to prove Lemma \ref{lemma:1.2} which gives the following explicit formula for the stabilization time of $\omega \in \{0,1\}^n$ 
satisfying $\omega_1 = 0$ and $\omega_n = 1$.
\[
T_{n}^{p}(\omega)=\frac{n}{2}+\max_{1\leq k\leq n}S_{k}-\frac{S_{n}}{2}-1
\]
\noindent

This formula, once known, can be proved directly by induction. The proof by induction, however, obscures the actual mechanics of the evolution. A more illuminating solution, which is how the formula was first discovered, comes from mapping binary strings to Young diagrams in a particular way. Under this interpretation, the evolution process becomes particularly simple and the explicit formula arises from the geometry of the Young diagram. This is the proof we will present in this section. We will begin by creating a map from $\{0,1\}^n$ to the space of Young diagrams.

\begin{definition} 
Let $\bY$ denote the collection of all finite sets $Y \subseteq \bN \times \bN$ which have the following property.
\[
(i,j) \in Y \Rightarrow (i',j') \in Y \mbox{ for all $i',j' \in \bN$ such that $i' \leq i$ and $j' \leq j$.}
\]
The empty set is also counted in $\bY$. It is customary to 
represent a set $Y \in \bY$ as a picture containing a ``collection of 
boxes'' (square boxes of side 1), where for every 
$(i,j) \in Y$ we take in our picture the box which has $(i,j)$
as its top right corner. In this guise, the set $Y$ is called
a {\em Young diagram}, see e.g. \cite{S1999}. The following example 
illustrates this representation.
\[
Y = \{ (1,1), (2,1), (3,1), (4,1), (1,2), (2,2), (1,3), (2,3) \} = \begin{aligned}\yng(2,2,4)\end{aligned}
\]
\end{definition}

\begin{definition} \label{def:2.2} For every $\omega \in \{0,1\}^{n}$ with $\omega_1 = 0$ and $\omega_n = 1$, let $U=U(\omega)$ be the number of $1$'s in the string $\omega$. 
Let $B$ be a $U\times(n-U)$ grid on the x-y plane. Reading $\omega$ from left to right, we construct a path starting from the top-left corner of $B$  by drawing a line 
horizontally to the right whenever we encounter a $0$ in $\omega$ and a line vertically downward whenever we encounter a $1$ in $\omega$. Since there are $U$ $``1"$'s 
and $n-U$ $``0"$'s to be found in $\omega$ we will get a path from the top-left corner in $B$ to the bottom-right corner in $B$. Here is an example of the path generated by the 
string $\omega = (0,1,1,0,1,0,1,1)$:
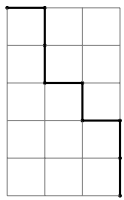
\begin{center}

\begin{tikzpicture}[scale=0.5]
\draw[fill] (3,0) circle (1pt);
\draw[fill] (3,1) circle (1pt);
\draw[fill] (3,2) circle (1pt);
\draw[fill] (2,2) circle (1pt);
\draw[fill] (2,3) circle (1pt);
\draw[fill] (1,3) circle (1pt);
\draw[fill] (1,4) circle (1pt);
\draw[fill] (1,5) circle (1pt);
\draw[fill] (0,5) circle (1pt);
\draw[help lines] (0,0) grid (3,5);
\draw[thick] (3,0) --(3,1)--(3,2)--(2,2)--(2,3)--(1,3)--(1,4)--(1,5)-- (0,5);
\end{tikzpicture}
\end{center}

\noindent
Define $\pi(\omega) \subset \mathbb{N} \times \mathbb{N}$ to be the unique path constructed in this manner. In our example, $\pi(0,1,1,0,1,0,1,1) = \{(0,5),(1,5),(1,4),(1,3),(2,3),(2,2),(3,2),(3,1),(3,0)\}$. The set of boxes under this path defines a Young diagram, which we denote $Y(\omega)$. In our example, this is the following diagram.
\[
Y(\omega) = \yng(1,1,2,3,3)
\]

\end{definition}

\begin{definition}
\label{def:2.3}
Let $Y \in \bY$ be a Young diagram and let $(i, j) \in Y$. We say that $(i,j)$ is an \emph{exposed corner} of $Y$ if $(i + 1, j) \not\in Y$ and $(i, j + 1) \not\in Y$. The 
\emph{corner cutting map} $\cut : \bY \to \bY$ is the map that removes all the exposed corners of a Young diagram as follows.
\[
\cut (Y) := Y \setminus \{ (i,j) \in Y \mid
(i, j) \text{ is an exposed corner of $Y$}\}
\]
It is easily verified by the definition that removing the exposed corners
does indeed yield another Young diagram. That is to say, $\cut$ is a well-defined function from $\bY \to \bY$.

\end{definition}

\begin{proposition}
Let $\omega \in \{0,1\}^{n}$ be a string and let $\omega^{\prime}$ be the string obtained from $\omega$ after one stage of the evolution, that is after replacing 
instances of $``01"$ with $``10"$ once. Then the Young diagram $Y(\omega^{\prime})$ is obtained by cutting the corners of the Young diagram $Y(\omega)$.
\[
Y(\omega^{\prime}) = \cut\left(Y(\omega)\right)
\]
\end{proposition}

\begin{proof}
Consider the path $\pi(\omega)$ defined in Definition \ref{def:2.2}. Any instances of $``01"$ in $\omega$ will correspond to a horizontal segment followed by a vertical segment, while instances of $``10"$ correspond to a vertical segment followed by a horizontal segment. As such, the evolution $``01" \to ``10$ will translate into the following pictorial evolution for the path $\pi(\omega)$.

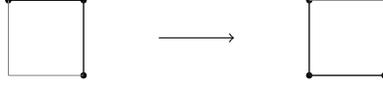
\begin{center}
\begin{tikzpicture}
\draw[fill] (1,0) circle (1pt);
\draw[fill] (1,1) circle (1pt);
\draw[fill] (0,1) circle (1pt);
\draw[fill] (5,0) circle (1pt);
\draw[fill] (4,1) circle (1pt);
\draw[fill] (4,0) circle (1pt);
\draw[help lines] (0,0) grid (1,1);
\draw[help lines] (4,0) grid (5,1);
\draw (1,0)--(1,1)--(0,1);
\draw[->] (2,0.5)--(3,0.5);
\draw (5,0)--(4,0)--(4,1);
\end{tikzpicture}
\end{center}

\noindent
These instances of $``01"$ correspond exactly to the exposed corners of $Y(\omega)$ since they have no neighbors above them or to their right. We therefore see that the substitution rule $``01" \mapsto ``10"$ amounts precisely to removing the exposed corners of $Y(\omega)$.
\end{proof}

\begin{definition}
Let $Y \in \bY$, $Y \not= \emptyset$ be a non-empty Young diagram. We define:
 \[
\depth (Y) = \max \{ i + j - 1 | (i, j) \in Y\}
\]
 We also set the convention $\depth(\emptyset) = 0$.
\end{definition}

\begin{proposition}
Let $Y \in \bY$, $Y \not= \emptyset$. Then $\depth(\cut(Y)) = \depth(Y) - 1$.
\end{proposition}

\begin{proof}
Suppose $(i, j) \in Y$ such that $i + j - 1 = \depth(Y)$. Clearly $(i + 1, j) \not\in Y$ and $(i, j + 1) \not\in Y$, since assuming otherwise contradicts $i + j - 1 = \depth(Y)$ is maximal. Then 
$(i, j)$ is an exposed corner, so $(i, j) \not\in \cut(Y)$. Hence $\depth(\cut(Y)) < \depth(Y)$.

Now, if either $(i-1,j) \in \cut(Y)$ or $(i,j-1)\in \cut(Y)$ then $\depth(\cut(Y)) \geq i+j-2 = \depth(Y) - 1$. By combining 
the inequalities, we conclude $\depth(\cut(Y)) = \depth(Y) - 1$. Otherwise, both $(i-1,j) \notin \cut(Y)$ 
and $(i,j-1)\notin \cut(Y)$. But this only happens in the case $Y = \{(1,1)\}$, in which case $\cut(Y)=\emptyset$ and 
the proposition holds by $\depth(\emptyset) = 0$. 
\end{proof}

\begin{proposition}
\label{prop:2.7}
For every $\omega \in \{0,1\}^{n}$ we have that:
\[
T_n^p(\omega) = \depth(Y(\omega))
\]
As in the introduction, $T_n^p(\omega)$ denotes the stabilization time of the string $\omega$
\end{proposition}
\begin{proof}
The proof follows by induction on $\depth(Y(\omega)$ using the previous two propositions. For the base case, if $\depth(Y(\omega)) = 0$, then 
$Y = \emptyset$ so $\omega$ is already stable and the stabilization time is 0 and the result holds. Now assume that for all $\omega$ with $\depth(Y(\omega)) = k - 1$, 
$T_n^p(\omega) = k-1$. Given any $\omega$ with $\depth(Y(\omega)) = k$, let $\omega^{\prime}$ be the evolution of $\omega$ by one step. Then, by the last 
two propositions we have that $\depth(Y(\omega^{\prime})) = \depth(\cut(Y(\omega))) = \depth(Y(\omega)) - 1  = k - 1$, and so by the induction hypothesis, 
$T_n^p(\omega^{\prime}) = k -1$. Since we have applied one step to get from $\omega$ to $\omega^{\prime}$, $T_n^p(\omega) = T_n^p(\omega^{\prime}) + 1 = k$. 
\end{proof}

\begin{proposition}
\label{prop:2.8}
Let $\omega \in \{0,1\}^n$, and for $0\leq k\leq n$ let $S_{k}=\sum_{i=1}^{k}(1 - 2\omega_{i})$. Let $U=U(\omega)$ be the number of $1$'s in the string $\omega$. 
In the case that $\omega_1 = 0$ and $\omega_n = 1$ we have the following explicit relationship between $S_k$ and the depth of the Young diagram 
$Y(\omega)$:
\[
\depth(Y(\omega))=U + \max_{0\leq k\leq n}S_{k} - 1
\]
\end{proposition}

\begin{proof}
Let $\pi ( \omega )$ be the path associated to
$\omega$ in Definition \ref{def:2.2}. An elementary computation
shows that $S$ is related to $\pi ( \omega )$ by the following
formula:
\[
\{ (k,S_k) : 0 \leq k \leq n \} =
\{ (i+(U-j), i-(U-j)) : (i,j) \in \pi ( \omega ) \}.
\]
This formula merely says that (because of the specifics of how each of $S$ and $\pi ( \omega )$ are constructed from $\omega$) the set of lattice points $S( \omega ) := \{ (k,S_k) : 0 \leq k \leq n \}$ is obtained out of $\pi ( \omega )$ via $45^o$ degree rotation and dilation by $\sqrt{2}$. The verification of the formula is left as exercise to the reader. An illustration of how $\pi ( \omega )$ and $S( \omega )$ look in a concrete case is shown in the next picture, drawn for $\omega = (0,0,1,1,0,1,1,1,0,1) \in \{ 0,1 \}^{10}$.

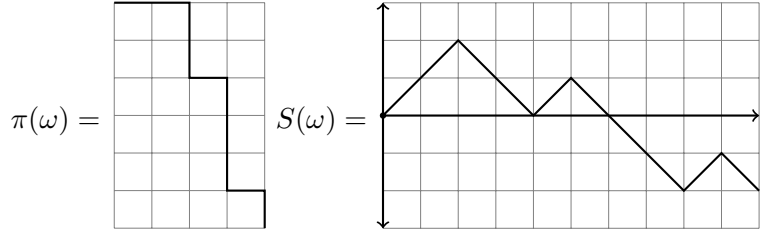
\[
\pi(\omega) = 
\begin{aligned}
\begin{tikzpicture}[scale=0.5]
\draw[help lines] (0,0) grid (4, 6);
\draw[thick] (0,6)--(2,6)--(2,4)--(3,4)--(3,1)--(4,1)--(4,0);
\end{tikzpicture}\\
\end{aligned} \text{     }S(\omega) =
\begin{aligned}
\begin{tikzpicture}[scale=0.5]
\draw[help lines] (0,-3) grid (10, 3);
\draw[thick] (0,0)--(2,2)--(4,0)--(5,1)--(8,-2)--(9,-1)--(10,-2);
\draw[thick, <->] (0,3) -- (0,0) -- (10,0);
\draw[thick, ->] (0,0) -- (0,-3);
\draw[fill] (0,0) circle (2pt);
\end{tikzpicture}\\
\end{aligned}
\]

\noindent
Now, if $\omega_{1} = 0$ and $\omega_{n} = 1$, we know that the boundary of the Young diagram $Y(\omega)$ is precisely the path $\pi(\omega)$. (The fact that $\omega_{1} = 0$ and $\omega_{n} = 1$ is needed here because otherwise there are some points in $\pi(\omega)$ which are not included in $Y(\omega)$). Since the $\depth(Y(\omega))$ is achieved 
somewhere on its boundary, we have that:
\begin{eqnarray*}
\depth(Y(\omega)) & = & \max\left\{ i+j-1:(i,j)\in Y(\omega)\right\} \\
 & = & \max\left\{ i+j-1:(i,j)\in\pi(\omega)\right\} \\
 & = & U+\max\left\{ i-(U-j):(i,j)\in\pi(\omega)\right\}-1 \\
 & = & U+\max\left\{ S_{k}:0\leq k\leq n\right\} -1
\end{eqnarray*}
\noindent
Where the last equality follows from the map between $S$ and $\pi(\omega)$ described above.
\end{proof}

\noindent \emph{ Proof of Lemma \ref{lemma:1.2}.}
This follows immediately when we combine Proposition \ref{prop:2.7} with Proposition \ref{prop:2.8}, and also use the elementary relation $U = \frac{1}{2}\left(n - S_n\right)$. \hfill $\blacksquare$

\section{Limit Distribution in the Special Case}

In this section we aim to prove the aforementioned weak limits for the random variable $\tilde{T}_{n}^{p}$, the stabilization time for strings from the probability space $\tilde{\Omega}_n^p = \{ \omega \in \{0,1\}^n : \omega_1 = 0, \omega_n = 1\}$. We will use the result from Lemma \ref{lemma:1.2} connecting $\tilde{T}_n^p$ to the random walk associated with $\omega$,  $S_k = \sum_{i=1}^k (1-2\omega_i)$. 
\begin{lemma}
\label{lemma:3.1}
Let $X_1, X_2, \ldots, X_{n-2}$ be i.i.d random variables which take the value $-1$ with probability $p$ and $1$ with probability $1-p$. For $0\leq k\leq n-2$, let $W_k = \sum_{i=1}^k X_i$ be the random walk which takes the $X_i$'s as its steps. Then:
\[
\tilde{T}_{n}^{p}  \stackrel{d}{=} \frac{n}{2}+\max_{0\leq k\leq n-2}W_{k}-\frac{W_{n-2}}{2}
\]
\end{lemma}

\begin{proof}
Make the identification that $X_i \stackrel{d}{=} (1-2\omega_{i+1})$. Since $\omega_1 = 0$, $\omega_n = 1$, we have  that $S_{k+1} \stackrel{d}{=} 1+W_k$ for $0\leq k \leq n-2$ and $S_n = 1+W_{n-2} - 1$. From this, it is clear that $\max_{1\leq k\leq n}S_{k} = 1 + \max_{0\leq k\leq n-2}W_{k}$, and the result is then immediate from Lemma \ref{lemma:1.2}.
\end{proof}

Since $W_k$ is the sum of many i.i.d. random variables, we are in a position to use tools like the central limit theorem and Donsker's theorem to find the limit distribution.  We divide the remaining results into our three settings, when $p > \frac{1}{2}$, when $p = \frac{1}{2}$, and the threshold setting $p_n = \ahalf + \frac{\lambda/2}{\sqrt{n}}$

\begin{proposition}
\label{prop:3.2}
For $p > \frac{1}{2}$, we have that, as $n \to \infty$:
\[
\frac{\max_{0\leq k \leq n-2} W_k}{\sqrt{n}} \stackrel{\p}{\to} 0
\]
\end{proposition}

\begin{proof}
$\max_{0\leq k < \infty} W_k$ is the maximum height achieved at any time by a weighted random walk, which takes steps upwards with probability $1-p < \frac{1}{2}$ and steps downwards with probability $p > \frac{1}{2}$. It is a result from elementary probability that this maximum is distributed like a \emph{geometric random variable} related to the parameter $q = \frac{1 - \sqrt{1-4p(1-p)}}{2p} < 1$, namely $\max W_k \stackrel{d}{=} Geom(1 - q) - 1$.  Since this random variable is finite almost surely, the result of the proposition is immediate from Markov's inequality.

\end{proof}

\begin{lemma}
\label{lemma:3.3}
In the setting $p > \frac{1}{2}$, $\tilde{T}_{n}^{p}$  has the following weak convergence to a Gaussian as $n\to\infty$:

\[
\frac{\tilde{T}_{n}^{p}-np}{\sqrt{n}}\Rightarrow N(0,p(1-p))
\]
\end{lemma}
\begin{proof}
We have, using the results of Lemma \ref{lemma:3.1} and Proposition \ref{prop:3.2}, that:

\begin{eqnarray*}
\frac{\tilde{T}_{n}^{p}-np}{\sqrt{n}}  &=& \frac{ \left( \tilde{T}_{n}^{p}-\frac{n}{2} \right) - \frac{n}{2}\left(2p -1\right)}{\sqrt{n}} \\
&\stackrel{d}{=}& \frac{ \left( \max_{0\leq k\leq n-2}W_{k}-\frac{W_{n-2}}{2} \right) - \frac{n}{2}\left(2p -1\right)}{\sqrt{n}} \\
&=& \frac{\max_{0\leq k \leq n-2} W_k}{\sqrt{n}} - \frac{1}{2}\frac{W_{n-2} - n(1-2p)}{\sqrt{n}} \\
&\To&  0 + N(0,p(1-p))
\end{eqnarray*}
Since $\frac{\max_{0\leq k\leq n-2}W_{k}}{\sqrt{n}} \To 0$ by Proposition \ref{prop:3.2} and since $W_{n-2}$ is the sum of the i.i.d. random variables $X_i$ with mean  $1-2p$ and variance $4p(1-p)$, so by the Central Limit Theorem, we have weak convergence to a Gaussian: $\frac{W_{n} - n(1-2p)}{\sqrt{n}} \To N(0,4p(1-p))$.
\end{proof}

\begin{lemma}
\label{lemma:3.4}
Let $B_t, t\in [0,1]$ be a Brownian motion and let $M_t = \max_{s\leq t} B_s$ be its running maximum. Then, in the case that $p =\frac{1}{2}$ we have the following weak limit:
\[
\frac{\tilde{T}_n^\ahalf - \frac{n}{2}}{\sqrt{n}} \To M_1 - \ahalf B_1
\]
\end{lemma}
\begin{proof}
We know that $\tilde{T}_{n}^{p}  \stackrel{d}{=} \frac{n}{2}+\max_{0\leq k\leq n-2}W_{k}-\frac{W_{n-2}}{2}$, so it suffices (by replacing $n$ with $n+2$) to show that:
\[
 \frac{\max_{0\leq k\leq n}W_{k}-\ahalf W_{n}}{\sqrt{n}} \To M_1 - \ahalf B_1
\]

This is a  direct application of Donsker's theorem  which stipulates weak convergence of the random walk to a Brownian motion when treated as a piecewise linear function under the correct scaling. Specifically, this says that when $X_k = \pm 1$ with probability $\ahalf$ then the random walk $L^n : [0,1] \to \mathbb{R}$ by $L^n(t) = \frac{1}{\sqrt{n}}\left[\sum_{k=1}^{\lfloor n t \rfloor}X_k + \left(t - \frac{ \lfloor n t\rfloor}{n} \right) X_{\lfloor n t \rfloor + 1}\right]$ has $L^n(\cdot) \To B(\cdot)$ as $n\to \infty$ in the sense of weak convergence on $C[0,1]$ (see e.g. Section 8 of \cite{B00}). This is exactly the setting we are in with $L^n(\frac{k}{n}) = \frac{1}{\sqrt{n}} \sum_{i=1}^k X_i = \frac{1}{\sqrt{n}} W_k$. Now, let $h:C[0,1]\to\mathbb{R}$ by $h(f(\cdot))=\sup_{t\in[0,1]}f(t)-\frac{1}{2}f(1)$. This is a continuous function on $C[0,1]$ with the sup norm, and so it respects weak limits.  The convergence $\frac{\max_{0\leq k\leq n}W_{k}-\ahalf W_{n}}{\sqrt{n}} = h(L^n(\cdot)) \To h(B(\cdot)) = M_1 - \ahalf B_1$ is exactly what is desired. 

\end{proof}

\begin{lemma}
\label{lemma:3.5}
 Let $B^{\lambda}_t, t\in [0,1]$ be a Brownian motion with drift $\lambda$ and let $M^{\lambda}_t = \max_{s\leq t} B^{\lambda}_s$ be its running maximum. Then, in the threshold setting that $p_n =\frac{1}{2} + \frac{\lambda / 2}{\sqrt{n}}$ we have the following weak limit:
\[
\frac{\tilde{T}_n^{p_n} - \frac{n}{2}}{\sqrt{n}} \To M^{\lambda}_1 - \ahalf B^{\lambda}_1
\]
\end{lemma}

\begin{proof} Following the argument of the previous lemma, it suffices to show that the piecewise linear random walk  $L^n : [0,1] \to \mathbb{R}$ by $L^n(t) = \frac{1}{\sqrt{n}}\left[\sum_{k=1}^{\lfloor n t \rfloor}X_k + \left(t - \frac{ \lfloor n t\rfloor}{n} \right) X_{\lfloor n t \rfloor + 1}\right]$ has $L^n(\cdot) \To B^{\lambda}(\cdot)$ as $n\to \infty$. This is an exercise in the theory of diffusion processes. In the threshold setting, $X_j = 1$ with probability $\ahalf + \frac{\lambda / 2}{\sqrt {n}}$ and $X_j = -1$ with probability $\ahalf - \frac{\lambda / 2}{\sqrt{n}}$ so Donsker's theorem does not directly apply. Instead one can examine the generator of this Markov chain, $A_n f(x) = \int \left(f(y) - f(x)\right) \Pi_n (x,dy)$ where $\Pi_n(x,\cdot)$ is the probability density of $L^n(t+\frac{1}{n})$ given that $L^{n}(t) = x$. In our case, since $L^n(t+\frac{1}{n})= L^n(t) + \frac{1}{\sqrt{n}}X_k$ we see from the distribution of $X_k$ that: 
\[
A_n f(x) = \left(\ahalf + \frac{\lambda / 2}{\sqrt {n}}\right)\left(f(x + \frac{1}{\sqrt{n}}) - f(x)\right) +  \left(\ahalf - \frac{\lambda / 2}{\sqrt {n}}\right)\left(f(x - \frac{1}{\sqrt{n}}) - f(x)\right)
\]
 It is then easily verified using the definition of a derivative that $\lim_{n\to\infty} \frac{1}{1/n} A_n f(x) = \ahalf f^{\prime \prime}(x) + \lambda f^{\prime}(x)$. This is exactly the generator for a Brownian motion with drift $\lambda$! This is enough to conclude that $L^n(\cdot) \To B^\lambda(\cdot)$, see \cite{SV} Section 11.2.
\end{proof}

\begin{lemma}
\label{lemma:3.6}
$M_1 - \ahalf B_1 \stackrel{d}{=} \halfchithree $ with probability density function:
\[
d\halfchithree = \frac{8\sqrt{2}}{\sqrt{\pi}}x^{2}\exp\left(-2x^{2}\right) dx \text{ for } x>0
\]
\end{lemma}
\begin{proof}
We verify this by computing the density of $M_1 - \ahalf B_1$. This is just a computation using the joint density for Brownian motion and its maximum, which is readily calculated using the reflection principle -- see for instance \cite{KS98} pg 95. The joint density function for Brownian motion and its maximum is:
\[  
\rho\left(M_{T}=b,B_{T}=a\right)=\frac{2\left(2b-a\right)}{\sqrt{2\pi}T^{\frac{3}{2}}}\exp\left(-\frac{\left(2b-a\right)^{2}}{2T}\right)
\]
 for $b>a$, $b>0$, and $0$ otherwise. Now to get the density function for $M_1 - \ahalf B_1$, one just integrates the joint density for $B_t$ and $M_t$ along a line:

\begin{eqnarray*}
\rho \left(M_1 -\ahalf B_1 = x\right) &=& \intop_{-2x}^{2x}\rho\left(M_{1}=\frac{y}{2}+x,B_{1}=y\right)dy\\
 & = & \sqrt{\frac{2}{\pi}}\intop_{-2x}^{2x}\left((y+2x)-y\right)\exp\left(-\frac{\left((y+2x)-y\right)}{2}^{2}\right)dy\\
 & = & \frac{8\sqrt{2}}{\sqrt{\pi}}x^{2}\exp\left(-2x^{2}\right)
\end{eqnarray*}
\end{proof} 

\begin{lemma}
\label{lemma:3.7}
 $M^{\lambda}_1 - \ahalf B^{\lambda}_1 \stackrel{d}{=} \nu_\lambda $ with probability density function:
\[
d\nu_\lambda(x)= \frac{4\sqrt{2}}{\lambda \sqrt{\pi}} e^{-\ahalf \lambda^2} \sinh (2 \lambda x) x e^{-2x^{2}} dx \text{ for } x>0
\]
\end{lemma}
\begin{proof}
By the Girsanov theorem, the  Brownian motion with drift $B^{\lambda}_t$ is absolutely continuous with respect to the drift free Brownian motion $B_t$ and the measures are related by the likelihood function $\exp \left(\lambda B_1 - \ahalf \lambda^2 \right)$ -- see for instance \cite{KS98} pg 190. Since this likelihood function depends only on the final position $B_1$, the calculation proceeds in exactly the same way as in Lemma \ref{lemma:3.6}, with this additional factor under the integral sign:
\begin{eqnarray*}
\rho \left(M^{\lambda}_1 -\ahalf B^{\lambda}_1 = x\right) &=& \intop_{-2x}^{2x}\rho\left(M_{1}=\frac{y}{2}+x,B_{1}=y\right)\exp \left(\lambda y - \ahalf \lambda^2 \right) dy\\
 & = & \sqrt{\frac{2}{\pi}}\intop_{-2x}^{2x}\left(2x\right)\exp\left(-\frac{(2x)^{2}}{2}\right) \exp \left(\lambda y - \ahalf \lambda^2 \right) dy\\
 & = & \frac{4\sqrt{2}}{\lambda \sqrt{\pi}} e^{-\ahalf \lambda^2} \sinh (2 \lambda x) x e^{-2x^{2}} 
\end{eqnarray*}
\end{proof}

\emph{Proof of Lemma \ref{lemma:1.3}}
In the setting $p > \ahalf$, the required convergence follows from Lemma \ref{lemma:3.3}. Likewise, the case $p = \ahalf$ is covered by Lemmas \ref{lemma:3.4} and \ref{lemma:3.6}, while the threshold setting $p_n = \ahalf + \frac{\lambda/2}{\sqrt{n}}$ is covered by Lemmas \ref{lemma:3.5} and \ref{lemma:3.7} \hfill $\blacksquare$

\section{Limit Distribution in the General Case}

So far, we have results in the special case that the string $\omega$ has $\omega_1 = 0$ and $\omega_n = 1$. In this section, we will bootstrap off of these results to see that we have the same limit in the general case where there is no restriction on $\omega_1$ or $\omega_n$.

\begin{lemma}
\label{lemma:4.1}
Let $\tilde{\mu}_{r}^{p}$ be the law of the special case random variable
$\tilde{T}_{r}^{p}$ and let $\mu_{n}^{p}$ be the law of the random
variable $T_{n}^{p}$. For $p \neq \frac{1}{2}$ the measures are related by:
\[
\mu_{n}^{p}=\frac{p^{n+1}-(1-p)^{n+1}}{2p-1}\delta_{0}+\sum_{r=0}^{n-2}\frac{(1-p)p^{n-r}-p(1-p)^{n-r}}{2p-1}\tilde{\mu}_{r+2}^{p}
\]
\end{lemma}

\begin{proof}
We begin by splitting the space $\{ 0, 1 \}^n$ into disjoint subsets. For stable $\omega \in \{ 0, 1 \}^n$, all ``1"s in $\omega$ must lie to the left of all ``0"s in $\omega$. Consequently, there are $n+1$ stable strings in $\{ 0, 1 \}^n$ for which the time to stabilization is zero. For each $1 \le i \le n$, there is precisely one such string with exactly 
$i$ ``1"s. These strings contribute the following value to $\mu_{n}^{p}$.

\begin{equation}
\sum_{i=0}^{n} p^{i}(1-p)^{n-i}\delta_{0} =  \frac{p^{n+1} - (1-p)^{n + 1}}{2p-1}\delta_{0}
\end{equation}

\noindent
For non-stable $\omega \in \{ 0, 1 \}^n$, we will introduce $r$, the number of elements of $\omega$ lying between the first ``0" and the last ``1", so that we may write $\omega = (1, ... ,1,0,x_1, ... ,x_{r},1,0, ... ,0)$, with $x \in \{ 0, 1 \}^{r}$. The time until stabilization of $\omega$ is the time until stabilization of $(0,x_1,x_2...x_{r},1)$. For given $r$, the distribution of these times is precisely distributed like $\tilde{T}_{r+2}^{p}$ because this string is in the special case. These strings contribute the following value to $\mu_{n}^{p}$ for a given $r$:

\begin{equation}
\sum_{i=0}^{n - r - 2} p^{i+1}(1-p)^{n-r-i-1} \tilde{\mu}_{r+2}^{p} =  \frac{(1-p)p^{n-r} - p(1-p)^{n-r}}{2p-1} \tilde{\mu}_{r+2}^{p}
\end{equation}

\noindent
By summing over all possible values of $r$, we get the desired result.
\end{proof}

\begin{lemma}
\label{lemma:4.2}
When $p = \frac{1}{2}$, we have:
\[
\mu_{n}^{\frac{1}{2}}=\frac{n+1}{2^{n}}\delta_{0} + \sum_{r=0}^{n-2} \frac{n-r-1}{2^{n-r}} \tilde{\mu}_{r+2}^{\ahalf}
\]
\end{lemma}
\begin{proof} 
The proof is the same as above. The only difference is that in the case $p = \ahalf$ the sums we need to evaluate are arithmetic sums, instead of geometric ones. 
\end{proof} 

\begin{notation}
We introduce the notation $C_{r,n}^p$, a set of positive coefficients defined by the results of  Lemma \ref{lemma:4.1} and \ref{lemma:4.2}, so that the 
following holds true for all values of $p$:
\[
\mu_{n}^{p} = C_{0,n}^{p} \delta_0 + \sum_{r=2}^{n} C_{r,n}^{p} \tilde{\mu}_{r}^{p}
\]
\noindent
The reason for doing this is to unify the notation from the cases $p = \ahalf$ and $p \neq \ahalf$. To prove the result we are after, we only need three properties of the coefficients $C_{r,n}^p$. These are proven in the next lemma.
\end{notation}

\begin{lemma}
\label{lemma:4.4}
For each $n$,  let $\theta_n = \lfloor n^\frac{1}{4} \rfloor$. (In fact, any $\theta_n = \lfloor n^\alpha \rfloor$ where $0 < \alpha < \ahalf$ will work). Given $\theta_{n}$ as defined above, we have the following properties of $C_{r,n}^{p}$:
\[  \begin{array}{lccc}
 \text{1.}&  C_{0,n}^p + \displaystyle\sum_{r=2}^{n} C_{r,n}^{p} &=& 1 \\
 \text{2.}& C_{0,n}^p + \displaystyle\sum_{r=2}^{n-\theta_n} C_{r,n}^{p} & \stackrel{n\to\infty}{\longrightarrow} & 0 \\
 \text{3.}& \displaystyle\sum_{r=n - \theta_n+1}^{n} C_{r,n}^{p} & \stackrel{n\to\infty}{\longrightarrow} & 1 
\end{array} \]
 The same properties still hold in the threshold setting where $p = p_n$.
\end{lemma}

\begin{proof} Property 1 holds since the $C_{r,n}^p$s arose in Lemma \ref{lemma:4.1} and Lemma \ref{lemma:4.2} as the division of the probability space $\Omega_{n}^{p}$ into disjoint pieces.  Once 1 is established, 2 and 3 are equivalent. To establish 2, we use the definitions of $C_{r,n}^p$ which are found in Lemma \ref{lemma:4.1} and Lemma \ref{lemma:4.2}. For  $p > \ahalf$, we have:

\begin{eqnarray*}
C_{0,n}^p + \sum_{r=2}^{n-\theta_n} C_{r,n}^{p} &=& \frac{p^{n+1} - (1-p)^{n + 1}}{2p-1} + \sum_{r=2}^{n-\theta_n} \frac{(1-p)p^{n-r+2} - p(1-p)^{n-r+2}}{2p-1}\\
& \leq & \frac{p^{n+1}}{2p - 1} \cdot 1 + \frac{1-p}{2p-1} \sum_{r=2}^{n-\theta_n} p^{n-r+2} \cdot 1\\
&=& \frac{p^{\theta_n+2}}{2p-1} \stackrel{n \to \infty}{\longrightarrow} 0
\end{eqnarray*}
\noindent
The last line holds since $\theta_n \to \infty$ as $n\to\infty$.  This estimate also works to see that $C_{0,n}^{p_n} + \displaystyle\sum_{r=2}^{n-\theta_n} C_{r,n}^{p_n} \to 0$  as $n \to \infty$ in the threshold setting $p_n = \ahalf + \frac{\lambda / 2}{\sqrt{n}}$. In this case, $\frac{p_n^{\theta_n+2}}{2p_n-1}$ is an indeterminant ``$0/0$" limit. Elementary methods, however, show that the limit is zero when $\theta_n$ is a power of $n$ as we have chosen here.  Finally, when $p = \ahalf$, we have:

\begin{eqnarray*}
 C_{0,n}^p + \sum_{r=2}^{n-\theta_n} C_{r,n}^{p} &=& \frac{n+1}{2^{n}}+\sum_{r=2}^{n-\theta_n}\frac{n-r+1}{2^{n-r+2}} \\
 & = & \frac{n+1}{2^n}+\left(\frac{\theta_n+2}{2^{\theta_n+1}}-\frac{n+1}{2^{n}} \right) \stackrel{n \to \infty}{\longrightarrow} 0
\end{eqnarray*}

\end{proof}

\begin{proposition}
\label{proposition:4.5}
The random variable $T_n^{p}$ converges in the same way that $\tilde{T}_{n}^{p}$ converges as established in Lemma \ref{lemma:1.3}. That is $(T_{n}^{p} - pn)/\sqrt{n} \To N(0,p(1-p))$ for $p>\ahalf$, $(T_{n}^{p} - \ahalf n)/\sqrt{n} \To \halfchithree$ for $p = \ahalf$ and $(T_{n}^{p_n} - \ahalf n)/\sqrt{n} \To \nu_\lambda$ in the threshold setting where $p=p_n$.
\end{proposition}

\begin{proof}
The proof of the three separate settings, ($p=\ahalf$, $p \neq \ahalf$, and the threshold setting $p=p_n$ ) are all handled by the same argument. To be concrete, we will focus on the threshold setting and show $(T_{n}^{p_n} - \ahalf n)/\sqrt{n} \To \nu_\lambda$. The argument proceeds in the same way for all three settings because Lemma \ref{lemma:4.4} holds for all three settings. 

We use the following characterization of weak convergence for distributions
$\rho_{n}$: $\rho_{n}\To\rho$ if and only if $\rho_{n}\left((-\infty,a]\right)\to\rho\left((-\infty,a]\right)$
for every $a\in\mathbb{R}$ with $\rho(\{a\})=0$. (see e.g. \cite{B00} example
2.3 pg 18.) Since $\nu_\lambda$ has no atoms, our aim is to show that for every $a \in \mathbb{R}$:
\[
\left|\p\left(\frac{T^{p_n}_n-\ahalf n}{\sqrt{n}}\in \left( \infty, a\right]\right)-\nu_\lambda\left( \left(-\infty,a\right] \right)\right| \stackrel{ n \to \infty}{\longrightarrow} 0
\]

\noindent
For the remainder of the proof, fix an arbitrary $a \in \mathbb{R}$ and for convenience denote $A=(-\infty,a]$. We proceed by dividing into three terms using the triangle inequality and the set-algebra identity that $\sqrt{n}A+\ahalf n  = \left(\sqrt{r}A+\ahalf r\right)\dot{\bigcup}\left(\sqrt{r}a+\ahalf r,\sqrt{n}a+\ahalf n\right]$ which holds for every $r \in [0,n]$:

\begin{eqnarray*}
\left|\p\left(\frac{T^{p_n}_n-\ahalf n}{\sqrt{n}}\in A\right)-\nu_\lambda(A)\right| & = & \left|\left(\mu^{p_n}_n\right)\left(\sqrt{n}A+\ahalf n\right)-\nu_\lambda(A)\right|\\
&=& \left|\left(C^{p_n}_{0,n}\delta_0 + \sum_{r=2}^{n} C^{p_n}_{r,n}\tilde{\mu}^{p_n}_{r} \right)\left(\sqrt{n}A+\ahalf n\right)-\nu_\lambda(A)\right|\\
&\leq & \left|\left(C^{p_n}_{0,n}\delta_0 + \sum_{r=2}^{n-\theta_n} C^{p_n}_{r,n}\tilde{\mu}^{p_n}_{r} \right)\left(\sqrt{n}A+\ahalf n\right)\right|\\
& & + \left|\left(\sum_{r= n -\theta_n + 1}^{n} C^{p_n}_{r,n}\tilde{\mu}^{p_n}_{r} \right)\left(\sqrt{r}A+\ahalf r\right)-\nu_\lambda(A)\right|\\
& & + \left|\left(\sum_{r=n -\theta_n + 1}^{n} C^{p_n}_{r,n}\tilde{\mu}^{p_n}_{r} \right)\left(\sqrt{r}a+\ahalf r,\sqrt{n}a+\ahalf n\right]\right|
\end{eqnarray*}
\noindent
We now show that each term individually goes to zero as $n \to \infty$. Each term is handled in a separate claim.

\begin{claim1}
\[
\left|\left(C^{p_n}_{0,n}\delta_0 + \sum_{r=2}^{n-\theta_n} C^{p_n}_{r,n}\tilde{\mu}^{p_n}_{r} \right)\left(\sqrt{n}A+\ahalf n\right)\right| \stackrel{n\to\infty}{\longrightarrow} 0
\]
\end{claim1}
\begin{claimpf}
This is immediate from the second property in Lemma \ref{lemma:4.4}:

\[
\left|\left(C^{p_n}_{0,n}\delta_0 + \sum_{r=2}^{n-\theta_n} C^{p_n}_{r,n}\tilde{\mu}^{p_n}_{r} \right)\left(\sqrt{n}A+\ahalf n\right)\right| \leq \left| C^{p_n}_{0,n} \cdot 1 + \sum_{r=2}^{n-\theta_n} C^{p_n}_{r,n} \cdot 1 \right| \stackrel{n\to\infty}{\longrightarrow} 0
\]

\hfill $\square$
\end{claimpf}

\begin{claim2}
\[
\left|\left(\sum_{r= n -\theta_n + 1}^{n} C^{p_n}_{r,n}\tilde{\mu}^{p_n}_{r} \right)\left(\sqrt{r}A+\ahalf r\right)-\nu_\lambda(A)\right| \stackrel{n\to\infty}{\longrightarrow} 0
\]
\end{claim2}
\begin{claimpf}
This follows from the results of Lemma \ref{lemma:4.4} and from the weak convergence for $\tilde{T}_n$, namely $\frac{\tilde{T}^{p_n}_n - \ahalf n}{\sqrt{n}} \To \nu_\lambda$. Fix $\epsilon > 0$. We use the weak convergence of $\tilde{T}^{p_n}_n$ applied to the set $A$ to find an $N \in \mathbb{N}$ so large so that $\left|\tilde{\mu}^{p_n}_{r}(\sqrt{r}A + \ahalf r) -\nu_\lambda(A)\right| < \epsilon$ for all $r > N-\theta_N+1$. Then, for all $n > N$ we have that:

\begin{eqnarray*}
\left|\sum_{n -\theta_n + 1}^{n} C^{p_n}_{r,n}\tilde{\mu}^{p_n}_{r} \left(\sqrt{r}A+\ahalf r\right)-\nu_\lambda(A)\right| &\leq & \left|\sum_{n -\theta_n + 1}^{n} C^{p_n}_{r,n} \left( \tilde{\mu}^{p_n}_{r}\left(\sqrt{r}A+\ahalf r\right)-\nu_\lambda(A)\right)\right|\\
 & & + \left| \left( C^{p_n}_{0,n} + \sum_{ 2}^{n-\theta_n} C^{p_n}_{r,n}\right) \nu_\lambda(A)\right|\\
 &\leq & \left(\sum_{n-\theta_n+1}^{n} C^{p_n}_{r,n}\right) \epsilon +  \left( C^{p_n}_{0,n} + \sum_{ 2}^{n-\theta_n} C^{p_n}_{r,n}\right) \nu_\lambda(A)\\
 &\to& 1\cdot \epsilon + 0 \cdot \nu_\lambda(A) = \epsilon 
\end{eqnarray*}

\noindent
where the limits are from Lemma \ref{lemma:4.4}. Since $\epsilon$ arbitrary, we have the result of the claim. \hfill $\square$
\end{claimpf}

\begin{claim3}
\[
\left|\left(\sum_{r=n -\theta_n + 1}^{n} C^{p_n}_{r,n}\tilde{\mu}^{p_n}_{r} \right)\left(\sqrt{r}a+\ahalf r,\sqrt{n}a+\ahalf n\right]\right| \stackrel{n\to\infty}{\longrightarrow} 0
\]
\end{claim3}
\begin{claimpf}
Fix an $\epsilon > 0$. Since the measure $\nu_\lambda$ has no atoms, choose a $\delta > 0$ so small so that $\nu_\lambda\left(a,a+\delta\right] < \epsilon$. Now, since $\frac{\tilde{T}^{p_n}_n - \ahalf n}{\sqrt{n}} \To \nu_\lambda$, choose $N$ so large so that for all $n > N-\theta_N + 1$ we have $ \left| \p \left( (\tilde{T}^{p_n}_n - \ahalf n) / \sqrt{n} \in \left(a,a+\delta \right] \right) -\nu_\lambda \left(a,a+\delta\right] \right| < \epsilon$. Now choose $M$ so large, so that the following holds whenever $n > M$ and $n \geq r \geq n-\theta_n+1$:
\[
a\frac{\sqrt{n} - \sqrt{r}}{\sqrt{r}} + \ahalf \frac{n-r}{\sqrt{r}} < \delta
\]
Such an $M$ always exists because:
\begin{eqnarray*}
a\frac{\sqrt{n} - \sqrt{r}}{\sqrt{r}} + \ahalf \frac{n-r}{\sqrt{r}} &\leq & |a| \left( \sqrt{\frac{1}{1-\frac{\theta_n}{n}}} - 1 \right) + \ahalf \frac{\theta_n}{\sqrt{n-\theta_n}} \\
&\approx& |a| \frac{1}{2} \left(\frac{\theta_n}{n}\right) + \ahalf \left(\frac{\theta_n}{\sqrt{n}}\right) \to 0 \text{ since } \frac{\theta_n}{\sqrt{n}} \to 0
\end{eqnarray*}
\noindent

For $n$ larger than both $N$ and $M$, we will have then for every $r$ with $n \geq r \geq n-\theta_n+1$ that:
\begin{eqnarray*}
\tilde{\mu}^{p_n}_{r}\left(\sqrt{r}a+\ahalf r,\sqrt{n}a+\ahalf n \right] &=& \p \left( \frac{\tilde{T}^{p_n}_r - \ahalf r}{\sqrt{r}} \in \left(a,a+\left(a\frac{\sqrt{n} - \sqrt{r}}{\sqrt{r}} + \ahalf \frac{n-r}{\sqrt{r}}\right)\right] \right)\\
 &\leq & \p \left( \frac{\tilde{T}^{p_n}_r - \ahalf r}{\sqrt{r}} \in \left(a,a+\delta\right] \right)\\
 &\leq & \nu_\lambda \left(a,a+\delta \right] + \epsilon \leq 2\epsilon
\end{eqnarray*}
\noindent
Hence, for such $n$ we have: $\left|\left(\sum_{r=n -\theta_n + 1}^{n} C^{p_n}_{r,n}\tilde{\mu}^{p_n}_{r} \right)\left(\sqrt{r}a+\ahalf r,\sqrt{n}a+\ahalf n\right]\right| \leq 1\cdot2 \epsilon$.
\end{claimpf}

\end{proof}

\section{Acknowledgments}
 The second author would like to thank Professor S.R.S. Varadhan for helpful discussions that led to the development and proof of the calculation in the threshold setting.

Jacob Funk

Department of Pure Mathematics, University of Waterloo,

Waterloo, Ontario N2L 3G1, Canada

Email: jjfunk@uwaterloo.ca

\vspace{8pt}

Mihai Nica

Courant Institute of Mathematical Sciences

New York University

251 Mercer Street
New York, N.Y. 10012-1185
Mail Code: 0711

Email: nica@cims.nyu.edu

\vspace{8pt}

Michael Noyes

Department of Pure Mathematics, University of Waterloo,

Waterloo, Ontario N2L 3G1, Canada

Email: mnoyes@uwaterloo.ca

\end{document}